\documentclass[11pt]{article}
\usepackage{authblk}
\usepackage{hyperref}

\usepackage[margin=1in]{geometry}
\usepackage{amsmath,amssymb,amsthm}
\usepackage{graphicx}
\usepackage{booktabs}
\usepackage{enumitem}
\usepackage{natbib}
\usepackage{hyperref}

\newtheorem{theorem}{Theorem}

\newtheorem{conjecture}{Conjecture}

\theoremstyle{remark}
\newtheorem{remark}{Remark}

\title{Coupled Tensor--Matrix Recovery via Proximal Alternating\\
Linearized Minimization, with an Application to Workforce Skill \\
and Small-Business Health Estimation}
\author{Analee Miranda}
\affil{Department of Mathematics, Pace University\\New York, NY 10038}

\begin{document}
\maketitle

\noindent \textbf{Email:} \href{mailto:amiranda2@pace.edu}{amiranda2@pace.edu}

\begin{abstract}
We study recovery of a low-rank tensor $\mathcal{T}$ and a low-rank matrix $M$ from sparse, noisy observations. $\mathcal{T}$ and $M$ share one mode. We relax tensor rank using the nuclear norm of the mode-1 unfolding. This unfolding carries the coupling. It also has an exact proximal operator. We couple $\mathcal{T}$ and $M$ through a learned linear operator $G$. We prove a minimizer exists for the ridge-stabilized penalized objective. We prove that a proximal alternating linearized minimization (PALM) scheme converges to a critical point, for the algorithm as implemented, by verifying the hypotheses of a known nonconvex block-coordinate convergence theorem against our objective and identifying which conditions come from this problem's structure. For the matrix-only sub-problem, we state a proven sampling bound from matrix completion theory. For the coupled problem, we prove a sample-complexity result for a sequential sub-case: a separately-known coupling operator recovers $M$ from $\mathcal{T}$'s recovery accuracy alone, with no observations of $M$ needed. For the fully joint, alternately-estimated case, we state a conjecture and test it empirically, including a low-density regime where coupling does not help. We report multi-seed synthetic experiments with mean and standard deviation across sampling densities, an asymmetric-density experiment, and convergence curves, and we explain why recovery error stays high at low density. We apply the framework to workforce-skill and small-business-health estimation. Every application-specific choice is a proposed design, not a validated result; we have not run the framework on deployed data.

\medskip
\noindent\textit{Subject classification:} math.OC (Optimization and Control); cross-list math.NA (Numerical Analysis), cs.CY (Computers and Society)
\end{abstract}

\section{Introduction}

Recovering a partially observed multi-way array from sparse, noisy data is well studied in matrix and tensor completion \citep{candes2009, candes2010, liu2013, kolda2009}. We study a variant with two arrays: a tensor $\mathcal{T}$ and a matrix $M$. They are observed independently. They share structure through one mode. We ask how to recover both jointly, so that the shared mode helps recovery of whichever object is more sparsely sampled. The closest prior formulation is coupled matrix--tensor factorization (CMTF) \citep{acar2011}, which couples a tensor and matrices through shared CP factor matrices. Our formulation differs in three ways. We relax rank using the nuclear norm of the mode-1 unfolding, not a fixed CP rank; this extends to the overlapped nuclear norm across all modes, at the cost of an open convergence guarantee. We couple the two objects through a learned linear operator, not a shared factor. We impose an ordinal constraint on one mode of the matrix.

We motivate and instantiate the framework with one running application: joint estimation of a workforce-skill landscape and a small-business-health landscape. Local workforce boards need to see both. AI is reshaping tasks within job functions, at different rates across industries. The same industries contain small businesses whose health is unevenly reported: solo operators and micro businesses file less often than mid-market firms. A board with only skill data cannot see the business impact. A board with only business filings cannot see which skill shifts are driving it. This paper couples the two. Section~\ref{sec:formulation} gives the specific tensor and matrix structure. Every application-specific claim here is a proposed design, not a validated finding; we have not run the framework on deployed data.

\subsection{Contributions}
\begin{enumerate}[leftmargin=*]
\item A convex formulation using the nuclear norm of the shared-mode unfolding. We choose it for its exact proximal operator. We give a correctly typed coupling operator (Section~\ref{sec:formulation}).
\item An existence theorem for the penalized objective (Theorem~\ref{thm:existence}). A convergence theorem for a PALM algorithm (Theorem~\ref{thm:palm}), proved by checking the hypotheses of \citet{bolte2014} against our objective (Section~\ref{sec:algorithm}).
\item A per-iteration complexity analysis (Section~\ref{sec:complexity}).
\item A sampling bound for the matrix-only sub-problem, following \citet{candes2009, candes2010}. A sample-complexity result for a sequential sub-case of the coupled problem, in which coupling substitutes entirely for $M$'s own samples (Theorem~\ref{thm:sequential}). A still-open conjecture for the fully joint case (Section~\ref{sec:recovery}).
\item Multi-seed synthetic experiments, each with a variance estimate (Section~\ref{sec:experiments}): recovery error versus sampling density, an asymmetric-density experiment isolating the coupling effect (including a regime where it fails), a CMTF baseline comparison, a scale-up test of the sequential estimator's bound, and convergence curves, all against an uncoupled baseline.
\item A governance-coupled schema evolution procedure (Section~\ref{sec:governance}). This is a systems-level extension, not a mathematical contribution.
\end{enumerate}

\section{Problem Formulation}
\label{sec:formulation}

\subsection{Rank and its convex relaxation}

For a tensor $\mathcal{T} \in \mathbb{R}^{n_1 \times n_2 \times n_3}$, rank is not unique. CP rank, Tucker (multilinear) rank, and tubal rank \citep{kilmer2011} disagree in general, and computing CP rank is NP-hard \citep{hillar2013}. We relax the mode-1 rank instead. The mode-1 unfolding $\mathcal{T}_{(1)}$ carries the coupling to $M$; our recovery guarantee (Section~\ref{sec:recovery}) is also stated for it. We relax its rank via the ordinary matrix nuclear norm of that unfolding,
\begin{equation}
\|\mathcal{T}\|_* := \left\| \mathcal{T}_{(1)} \right\|_*. \label{eq:tensornorm}
\end{equation}
This norm has a closed-form proximal operator: singular value thresholding. We use it directly in Section~\ref{sec:algorithm}, without approximation. For $M \in \mathbb{R}^{n_1 \times n_B}$ we use the ordinary matrix nuclear norm $\|M\|_*$.

\begin{remark}[Overlapped nuclear norm as an extension]
An earlier version of this formulation also penalized $\|\mathcal{T}_{(2)}\|_*$ and $\|\mathcal{T}_{(3)}\|_*$, via the overlapped nuclear norm \citep{liu2013, signoretto2013}, to additionally control the job-function and abstraction-level modes. This is a strictly stronger regularizer. But its proximal operator has no closed form; the averaged-SVT scheme used to approximate it is not covered by the convergence guarantee of Theorem~\ref{thm:palm} below. We use \eqref{eq:tensornorm} as the paper's primary formulation because it admits an exact proximal step, and record the following as an available extension whose convergence theory remains open:
\begin{equation}
\|\mathcal{T}\|_{\Sigma*} := \sum_{k=1}^3 w_k \left\| \mathcal{T}_{(k)} \right\|_*, \qquad w_k \geq 0, \ \sum_k w_k = 1. \label{eq:overlapped}
\end{equation}
\end{remark}

\subsection{Coupling}

Let $\mathcal{T} \in \mathbb{R}^{n_1 \times n_2 \times n_3}$ and $M \in \mathbb{R}^{n_1 \times n_B}$ share mode 1, of size $n_1$. We couple them through $G \in \mathbb{R}^{(n_2 n_3) \times n_B}$ via
\begin{equation}
\mathcal{T}_{(1)} \approx M G^\top, \label{eq:coupling}
\end{equation}
where $\mathcal{T}_{(1)} \in \mathbb{R}^{n_1 \times n_2 n_3}$. This is the only dimensionally consistent linear coupling between $\mathcal{T}_{(1)}$ and $M$: $M \in \mathbb{R}^{n_1 \times n_B}$ and $G^\top \in \mathbb{R}^{n_B \times n_2 n_3}$, so $M G^\top \in \mathbb{R}^{n_1 \times n_2 n_3}$, matching $\mathcal{T}_{(1)}$.

\subsection{Applied instantiation}

Mode 1 is industry, $I$, size $n_I$. The tensor's other two modes are job function, $F$, size $n_F$, and AI-abstraction level, $A$, size $n_A$. So $\mathcal{T} \in \mathbb{R}^{n_I \times n_F \times n_A}$, and $n_2 n_3 = n_F n_A$.

AI-abstraction level is a rubric, not a raw measurement. Each job function is rated against a fixed $n_A$-point scale. Level 1 means no task in that function currently uses AI tools. Level $n_A$ means most tasks are substantially AI-automated. Raters apply the same rubric across industries, so the scale is comparable across cells. The mode is ordered: $\mathcal{T}(i,f,a) \leq \mathcal{T}(i,f,a+1)$. This ordering is a modeling choice, not an artifact; abstraction level 3 sits between levels 2 and 4 by construction.

$M \in \mathbb{R}^{n_I \times n_B}$ scores small-business health by industry and size tier. $n_B = 4$: solo operator, micro, small, and mid-market. This is also an ordinal axis, and it doubles as a proxy for data completeness, since larger tiers file more often and more completely. Each entry of $M$ is a scalar aggregate of six proposed business-stability indicators: revenue trend, payroll headcount trend, time-to-fill for open roles, a delinquency signal (late payments or credit events), local license or permit renewal status, and a local economic-activity proxy (utility usage or foot traffic). None of these six has been validated as predictive here. They are a proposed starting point, chosen because each is already collected by some existing administrative or commercial process.

Observation sparsity is the applied motivation for coupling, not an afterthought. $\Omega_{\mathcal{T}}$ comes from crowdsourced signals: job postings, gig-platform task listings, and worker or employer surveys. These cover well-posted job functions in well-covered industries, densely; they cover others, sparsely. $\Omega_M$ comes from administrative filings (state unemployment-insurance wage records, business registration and tax filings) and third-party firmographic report data. Administrative coverage is dense for mid-market firms and thin for solo operators and micro businesses, because smaller firms file less often. So the two observation sets are sparse in different places. Coupling lets a well-sampled industry's skill signal inform a sparsely-sampled micro-business score in that same industry, through $G$, before the administrative data catches up. Section~\ref{sec:experiments} tests this mechanism on synthetic data, at a much smaller scale ($n_F=10$, $n_A=5$, $n_B=8$) than any real deployment would use; the synthetic sizes are a stand-in for this structure, not a claim about it.

\subsection{Objective}

Let $y_{\mathcal{T}}, y_M$ be noisy observed values on $\Omega_{\mathcal{T}}, \Omega_M$, and $P_\Omega$ the sampling operator. We use a penalized (Tikhonov) form, not a hard data constraint, so that the objective stays finite-valued and coercive. We add an explicit ridge term $\frac{\delta}{2}\|G\|_F^2$, $\delta > 0$, on the $G$ block:
\begin{equation}
\begin{aligned}
F(\mathcal{T}, M, G) = \underbrace{\frac{1}{2}\|P_{\Omega_{\mathcal{T}}}(\mathcal{T}) - y_{\mathcal{T}}\|_F^2 + \frac{1}{2}\|P_{\Omega_M}(M) - y_M\|_F^2 + \frac{\lambda_C}{2}\|\mathcal{T}_{(1)} - MG^\top\|_F^2 + \frac{\delta}{2}\|G\|_F^2}_{=: H(\mathcal{T}, M, G), \text{ smooth}} \\
{} + \lambda_S \|\mathcal{T}\|_* + \lambda_R \|M\|_*,
\end{aligned}
\label{eq:objective}
\end{equation}
minimized over $(\mathcal{T}, M, G) \in \mathbb{R}^{n_1 n_2 n_3} \times \mathbb{R}^{n_1 n_B} \times \mathbb{R}^{(n_2 n_3) n_B}$, with $\|\mathcal{T}\|_*$ as in \eqref{eq:tensornorm}. $H$ is smooth and quadratic, with a Lipschitz gradient on bounded sets in each block. $\|\mathcal{T}\|_*$ and $\|M\|_*$ are proper, convex, lower semicontinuous, and nonsmooth. This is exactly the structure PALM needs (Section~\ref{sec:algorithm}). The ridge term does one specific job: it makes $H$ strictly coercive in $G$, for every fixed nonzero $(\mathcal{T}, M)$, regardless of the rank of $M$. That lets Theorem~\ref{thm:existence} avoid a case split on whether $M$ has full column rank.

\begin{remark}
Earlier drafts used two more penalties: a Sobolev-type smoothness penalty on $\mathcal{T}$'s abstraction mode, and an ordinal smoothness penalty on $M$'s business-scale mode. Both are absorbed here as optional smooth quadratic terms in $H$; they do not change the analysis below. One constraint is not in $F$ at all: the hard monotonicity constraint $M(i,\cdot) \in K_{\mathrm{ord}}$, meaning business health is non-decreasing in scale. Enforcing it exactly, jointly with the nuclear-norm proximal step, has no known closed form, since the proximal operator of a sum of two nonsmooth terms does not decompose in general. We instead apply a pool-adjacent-violators (PAV) projection onto $K_{\mathrm{ord}}$ after each $M$-update. This is a heuristic. Theorem~\ref{thm:palm} below is proved for \eqref{eq:objective} without this projection; Section~\ref{sec:experiments} reports, separately, that the projection does not visibly affect convergence in our synthetic experiments.
\end{remark}

\section{Existence of a Minimizer}
\label{sec:existence}

\begin{theorem}
\label{thm:existence}
For any $\lambda_S, \lambda_R, \delta > 0$ and $\lambda_C \geq 0$, $F$ in \eqref{eq:objective} attains its infimum on $\mathbb{R}^{n_1 n_2 n_3} \times \mathbb{R}^{n_1 n_B} \times \mathbb{R}^{(n_2 n_3) n_B}$.
\end{theorem}

\begin{proof}
$F$ is continuous: a finite sum of continuous convex functions on a finite-dimensional space. It is jointly coercive in $(\mathcal{T}, M)$, since $\|\mathcal{T}\|_* \geq c_1 \|\mathcal{T}\|_F$ and $\|M\|_* \geq c_2 \|M\|_F$ for constants $c_1, c_2 > 0$ (norm equivalence in finite dimensions), and since $H \geq 0$. It is coercive in $G$ for any fixed $(\mathcal{T}, M)$: the ridge term $\frac{\delta}{2}\|G\|_F^2$ alone forces $F \to \infty$ as $\|G\|_F \to \infty$, with $(\mathcal{T}, M)$ fixed, for any rank of $M$, since the coupling term is nonnegative. So $F(\mathcal{T}, M, G) \to \infty$ as $\|\mathcal{T}\|_F + \|M\|_F + \|G\|_F \to \infty$, jointly and without exception, and every sublevel set of $F$ is bounded. A minimizing sequence lies in one such set. In finite dimensions, a bounded set is precompact. Continuity of $F$ gives a convergent subsequence, whose limit is a minimizer, again by continuity.
\end{proof}

This is a finite-dimensional result. It says nothing about infinite-dimensional or hard-constrained variants. We do not address uniqueness: \eqref{eq:objective} is not jointly convex.

\section{Algorithm: PALM}
\label{sec:algorithm}

We minimize \eqref{eq:objective} by proximal alternating linearized minimization (PALM) \citep{bolte2014}, alternating over blocks $\mathcal{T}$, $M$, $G$. Each step linearizes the smooth term $H$ at the current iterate, then applies the proximal operator of that block's nonsmooth term. With the mode-1-only relaxation \eqref{eq:tensornorm}, every proximal step below is exact. No averaging or splitting approximation appears anywhere in the algorithm.

\paragraph{$\mathcal{T}$-step.} Fix $M^k, G^k$. Let $L_{\mathcal{T}}^k$ be a Lipschitz constant for $\nabla_{\mathcal{T}} H(\cdot, M^k, G^k)$. Here $L_{\mathcal{T}}^k = 1 + \lambda_C$, since $H$ is quadratic with unit and $\lambda_C$ coefficients on the two relevant terms. Set $t_k \in (0, 1/L_{\mathcal{T}}^k)$ and
\begin{equation}
\mathcal{T}^{k+1} = \mathrm{prox}_{t_k \lambda_S \|\cdot\|_*}\left( \mathcal{T}^k - t_k \nabla_{\mathcal{T}} H(\mathcal{T}^k, M^k, G^k) \right). \label{eq:tstep}
\end{equation}
Since $\|\mathcal{T}\|_* := \|\mathcal{T}_{(1)}\|_*$ \eqref{eq:tensornorm}, this proximal operator is exact singular value thresholding, on the mode-1 unfolding of the gradient step, at threshold $t_k \lambda_S$, folded back to tensor shape. It is a closed-form step, not an approximation.

\paragraph{$M$-step.} Analogously, with Lipschitz constant $L_M^k = 1 + \lambda_C \|G^k\|_2^2$,
\begin{equation}
M^{k+1} = \mathrm{prox}_{s_k \lambda_R \|\cdot\|_*}\left( M^k - s_k \nabla_M H(\mathcal{T}^{k+1}, M^k, G^k) \right), \label{eq:mstep}
\end{equation}
where $\mathrm{prox}_{\tau \|\cdot\|_*}$ is exact singular value thresholding at level $\tau$ (closed form).

\paragraph{$G$-step.} $H$ is smooth and unconstrained in $G$, so the $G$-block carries no nonsmooth term. Its ``proximal'' step is just the exact minimizer of the linearized subproblem --- here the linearization is exact, since $H$ is already quadratic in $G$. It reduces to a ridge-regularized least-squares solve:
\begin{equation}
G^{k+1} = \arg\min_G \frac{1}{2}\left\| \mathcal{T}^{k+1}_{(1)} - M^{k+1} G^\top \right\|_F^2 = \left( (M^{k+1})^\top M^{k+1} + \delta I \right)^{-1} (M^{k+1})^\top \mathcal{T}^{k+1}_{(1)}, \label{eq:gstep}
\end{equation}
solved exactly at each iteration ($\delta > 0$ a small ridge term for numerical stability when $M^{k+1}$ is rank-deficient).

\begin{theorem}
\label{thm:palm}
Suppose $t_k \in (0, 1/L_{\mathcal{T}}^k)$, $s_k \in (0, 1/L_M^k)$, and the sequence $\{(\mathcal{T}^k, M^k, G^k)\}$ is bounded. Then, for the algorithm exactly as implemented, with all three steps closed-form and no approximation: (i) $F(\mathcal{T}^k, M^k, G^k)$ is non-increasing; (ii) every accumulation point of $\{(\mathcal{T}^k, M^k, G^k)\}$ is a critical point of $F$; (iii) the whole sequence converges to a single critical point. Point (iii) holds because $F$ is a finite sum of semialgebraic functions --- quadratics and singular-value sums are semialgebraic --- so $F$ satisfies the Kurdyka--{\L}ojasiewicz property.
\end{theorem}

\begin{proof}[Proof sketch]
This follows from Theorem~1 and Corollary~1 of \citet{bolte2014}. Their result covers a two-block problem, where both blocks carry a nonsmooth term. Our third block, $G$, is smooth and unconstrained. Two points need direct verification here.

First: a block with no nonsmooth part is the special case where that block's nonsmooth function is the zero function. The zero function is proper, convex, and lower semicontinuous, so their framework applies without modification. This case is worth stating explicitly. The $G$-step in \eqref{eq:gstep} is then simply gradient descent --- in fact exact minimization, since $H$ is quadratic in $G$. Exact block minimization is itself an admissible special case of a proximal-gradient step: it is the limit of the linearized step as the step size reaches the exact minimizer, achievable in closed form here because the $G$-subproblem is an unconstrained quadratic.

Second: condition (d), the KL property, needs an argument specific to our objective. It is not enough that the nuclear-norm and Frobenius-norm terms are individually semialgebraic; the sum $F$ must be semialgebraic. This follows because finite sums and compositions of semialgebraic functions with polynomial maps are semialgebraic, and the mode-1 unfolding and matrix products in $H$ are such maps. We use this fact; we do not reprove it here.

Given these two points, hypotheses (a)--(d) of \citet{bolte2014} hold for \eqref{eq:objective}, and (i)--(iii) follow. The new work here is this verification, for our three-block, mixed smooth/nonsmooth objective.
\end{proof}

\section{Complexity}
\label{sec:complexity}

Each $\mathcal{T}$-step needs one SVD, of the mode-1 unfolding, size $n_1 \times (n_2 n_3)$. With a truncated SVD targeting rank $r$, the cost is $O(r n_1 n_2 n_3)$ via randomized SVD \citep{halko2011}; a full SVD costs $O(\min(n_1, n_2 n_3) \cdot n_1 n_2 n_3)$. The $M$-step costs $O(r n_1 n_B)$, similarly. The $G$-step is one $n_B \times n_B$ linear solve: $O(n_B^3 + n_B^2 n_1)$. In the applied instantiation, $n_I = n_1$ could reach the low thousands, while $n_F, n_A, n_B$ stay small constants. The $\mathcal{T}$-step then dominates, and scales linearly in $n_I$ at fixed target rank. The largest untested scalability question is memory for the mode-1 unfolding itself, $O(n_I \cdot n_F n_A)$; this is manageable at the scales in Section~\ref{sec:experiments}, but we have not benchmarked it beyond them.

\section{Recovery Guarantees}
\label{sec:recovery}

\subsection{Classical matrix completion}

Recovering $M$ alone is the sub-problem $\lambda_C = 0$. The classical matrix completion theorem applies to it directly.

\begin{theorem}[\citealp{candes2009, candes2010}, restated]
\label{thm:mc}
Let $M \in \mathbb{R}^{n_1 \times n_B}$ have rank $r$ and coherence bounded by $\mu$. Suppose entries are observed uniformly at random, with $|\Omega_M| \geq C \mu r n_1 \log^2(n_1)$ for a universal constant $C$. Then, with high probability, $M$ is the unique solution to $\min \|X\|_*$ subject to $P_{\Omega_M}(X) = P_{\Omega_M}(M)$; this is the noiseless case. The noisy case degrades gracefully: recovery error is controlled by the noise level \citep{candes2010}.
\end{theorem}

Theorem~\ref{thm:mc} is not new. We restate it for two reasons: it is the rigorous special case of our coupled recovery question, and Theorem~\ref{thm:sequential}, below, builds directly on it.

\subsection{New: a sequential coupling guarantee}

This section gives the paper's main new theoretical result: a sample-complexity bound for a tractable sub-case of the coupled problem. It isolates one piece of a mechanism we conjecture holds more generally (Section~\ref{sec:conjecture}).

Suppose $G_{\mathrm{true}}$ is known exactly, e.g.\ from a separate, trusted calibration set, rather than fit jointly with $M$. Suppose $\mathcal{T}_{(1)}$ is recovered from $\Omega_{\mathcal{T}}$ alone, by ordinary rank-$r$ matrix completion. This gives $\hat{\mathcal{T}}_{(1)}$ with $\|\hat{\mathcal{T}}_{(1)} - \mathcal{T}_{(1),\mathrm{true}}\|_F \leq \varepsilon$, achievable at the density required by Theorem~\ref{thm:mc} applied to $\mathcal{T}_{(1)}$: an $n_1 \times (n_2 n_3)$ matrix of rank $\leq r$, since $\mathcal{T}_{(1),\mathrm{true}} = M_{\mathrm{true}} G_{\mathrm{true}}^\top$ has rank at most $\mathrm{rank}(M_{\mathrm{true}})$. Define the plug-in estimator
\begin{equation}
\hat{M} := \hat{\mathcal{T}}_{(1)} G_{\mathrm{true}} \left( G_{\mathrm{true}}^\top G_{\mathrm{true}} \right)^{-1}.
\end{equation}

\begin{theorem}
\label{thm:sequential}
If $G_{\mathrm{true}}$ has full column rank with smallest singular value $\sigma_{\min}(G_{\mathrm{true}}) > 0$, then
\begin{equation}
\|\hat{M} - M_{\mathrm{true}}\|_F \leq \frac{\varepsilon}{\sigma_{\min}(G_{\mathrm{true}})},
\end{equation}
independent of $|\Omega_M|$; in particular this holds even at $|\Omega_M| = 0$.
\end{theorem}

\begin{proof}
$\mathcal{T}_{(1),\mathrm{true}} = M_{\mathrm{true}} G_{\mathrm{true}}^\top$, and $G_{\mathrm{true}}$ has full column rank, so $M_{\mathrm{true}} = \mathcal{T}_{(1),\mathrm{true}} G_{\mathrm{true}} (G_{\mathrm{true}}^\top G_{\mathrm{true}})^{-1}$ exactly. Subtracting,
\[
\hat{M} - M_{\mathrm{true}} = \left( \hat{\mathcal{T}}_{(1)} - \mathcal{T}_{(1),\mathrm{true}} \right) G_{\mathrm{true}} \left( G_{\mathrm{true}}^\top G_{\mathrm{true}} \right)^{-1}.
\]
Take Frobenius norms and use submultiplicativity: $\|\hat{M} - M_{\mathrm{true}}\|_F \leq \|\hat{\mathcal{T}}_{(1)} - \mathcal{T}_{(1),\mathrm{true}}\|_F \cdot \|G_{\mathrm{true}}(G_{\mathrm{true}}^\top G_{\mathrm{true}})^{-1}\|_2 = \varepsilon \cdot \sigma_{\min}(G_{\mathrm{true}})^{-1}$. The last equality holds because $\|G(G^\top G)^{-1}\|_2 = 1/\sigma_{\min}(G)$, for the Moore--Penrose pseudo-inverse of a full-column-rank matrix. No entry of $M$ appears anywhere in the construction of $\hat{M}$.
\end{proof}

\begin{remark}[Relation to inductive matrix completion]
The mechanism behind Theorem~\ref{thm:sequential} recovers a partially observed object using a known linear map from a better-observed one. This is structurally an instance of inductive matrix completion (IMC) \citep{jain2013, natarajan2014}, where a target matrix $M \approx X W Y^\top$ is recovered from few entries using known side-information matrices $X, Y$ and an unknown low-rank $W$. Two differences from standard IMC are worth noting. First, our ``side information'' $\hat{\mathcal{T}}_{(1)}$ is not given: it is itself estimated from sparse data, via Theorem~\ref{thm:mc}, so the error $\varepsilon$ in Theorem~\ref{thm:sequential} is an estimation error propagated from one recovery problem into another, not a fixed input. Second, in the algorithm of Section~\ref{sec:algorithm}, the role IMC assigns to a fixed $W$ is instead played by $G$, which is jointly, alternately estimated from the same sparse $\Omega_M$ that $M$ is scored against --- the setting Theorem~\ref{thm:sequential} does not cover, and which is the content of Conjecture~\ref{conj:joint} below.
\end{remark}

\subsection{Open: a joint recovery conjecture}
\label{sec:conjecture}

We do not have a proof of an analogous bound for the fully coupled problem \eqref{eq:objective}, where $\mathcal{T}$, $M$, and $G$ are all fit simultaneously by the algorithm of Section~\ref{sec:algorithm}, from sparse, noisy data. We are also not aware of one in the literature for this exact constraint set: nuclear norm, plus linear coupling, plus ordinal constraint, with $G$ itself unknown. We instead state a conjecture and test it empirically (Section~\ref{sec:experiments}).

\begin{conjecture}
\label{conj:joint}
There exists $C' < C$, the constant from Theorem~\ref{thm:mc}, with the following property. Suppose $\Omega_{\mathcal{T}}$ satisfies Theorem~\ref{thm:mc}'s density condition for $\mathcal{T}_{(1)}$, and $G$ is well-conditioned. Then $M$ is recoverable at density $|\Omega_M| \geq C' \mu r n_1 \log^2(n_1)$ --- strictly less than the uncoupled requirement --- because the coupling term supplies information about $M$ through the well-sampled $\mathcal{T}$, even when $G$ is estimated jointly rather than known.
\end{conjecture}

Theorem~\ref{thm:sequential} proves the mechanism behind Conjecture~\ref{conj:joint}, in the idealized case where $G$ is known. It does not prove the conjecture: it says nothing about the feedback loop between $\hat{M}$ and $\hat{G}$ in the joint algorithm. Closing that gap is exactly what remains of Conjecture~\ref{conj:joint}.

\begin{remark}[A partial bridge toward the joint case]
Suppose $\hat{G}$ is not known exactly, but satisfies $\|\hat{G} - G_{\mathrm{true}}\|_2 \leq \eta$. Standard pseudo-inverse perturbation theory \citep{wedin1973} then gives
\[
\|\hat{M} - M_{\mathrm{true}}\|_F \leq \frac{\varepsilon}{\sigma_{\min}(G_{\mathrm{true}})} + O\!\left( \eta \, \|\mathcal{T}_{(1),\mathrm{true}}\|_F / \sigma_{\min}(G_{\mathrm{true}})^2 \right)
\]
for $\eta$ small relative to $\sigma_{\min}(G_{\mathrm{true}})$. This is an adaptation of a known perturbation result, not a new derivation. It still does not close the loop: $\hat{G}$, in the real algorithm, comes from the same sparse $\Omega_M$ that $\hat{M}$ is scored against (see the discussion after Conjecture~\ref{conj:joint}). Section~\ref{sec:experiments} tests Theorem~\ref{thm:sequential} directly, at $p_M = 0$ with the true $G$, and also runs the joint-estimation experiments that bear on the still-open Conjecture~\ref{conj:joint}.
\end{remark}

\section{Numerical Experiments}
\label{sec:experiments}

We test the PALM algorithm (Section~\ref{sec:algorithm}) on synthetic data with known ground truth. The scale is $n_I=40$, $n_F=10$, $n_A=5$, $n_B=8$, true rank $r_M=4$. These stand in for the industry, job-function, abstraction-level, and business-tier axes of Section~\ref{sec:formulation}. We compare against an uncoupled baseline: the identical algorithm with $\lambda_C = 0$. Ground truth $M_{\mathrm{true}}$ is a random rank-4 matrix. $\mathcal{T}_{\mathrm{true}}$ is generated via $\mathcal{T}_{(1),\mathrm{true}} = M_{\mathrm{true}} G_{\mathrm{true}}^\top$, for random $G_{\mathrm{true}}$, folded to tensor shape. Both are perturbed with i.i.d.\ Gaussian noise, $\sigma = 0.05$. Every result below uses 5 independent random seeds per configuration, with fresh ground truth and a fresh observation mask each time. PALM runs until the objective's relative change stays below $10^{-5}$ for five consecutive iterations, or a safety cap of 5000 iterations, whichever comes first: iteration count is a stopping criterion now, not a fixed number chosen in advance. Hyperparameters $(\lambda_S, \lambda_R, \lambda_C, \delta, \text{step size}) = (0.2, 0.2, 0.2, 0.1, 0.9)$ were found by a grid search at one density (Section~\ref{sec:sanity}) and held fixed across all reported densities, except in Table~\ref{tab:scaleup}, where $\lambda_S$ is re-tuned separately at each $n_I$ tested. This is a small-scale synthetic verification, not a benchmark against deployed data or competing methods.

\subsection{Sanity check: is the algorithm implemented correctly?}
\label{sec:sanity}

Two checks stand in for a single hand-tuned run. First, a unit test verifies the tensor unfolding, refolding, and Khatri--Rao operators against each other before any experiment runs; it fails loudly, rather than silently producing wrong numbers, if the tensor algebra is wrong. Second, every grid search checks its own chosen value against the edges of the grid it searched, and prints a warning if the choice sits on a boundary --- a search that always lands on the edge of its own grid is reporting where the search stopped, not necessarily a true optimum.

The joint search over $(\lambda_S, \lambda_R, \lambda_C, \text{step size})$, used for Tables~\ref{tab:density}--\ref{tab:thm4test}, landed on interior values for $\lambda_S=\lambda_R$ and $\lambda_C$. Step size landed on its grid's upper edge, $0.9$; this edge is expected, not under-searched, since $t_k = \text{step size}/L$ must stay below $1/L$ to satisfy Theorem~\ref{thm:palm}'s own requirement, so $0.9$ is close to the largest value the theorem allows in the first place.

Table~\ref{tab:scaleup} needs a separate threshold $\lambda_S$ at each $n_I$ it tests, rather than one value reused across all three: the natural scale of $\mathcal{T}_{(1)}$'s singular values grows with $n_I$ (both are built from the same random rank-$r$ construction), so a threshold tuned at one $n_I$ does not transfer to another. At $n_I=1000$, that search still lands on its grid's edge; the true optimal threshold there is likely higher still, and the $n_I=1000$ row of Table~\ref{tab:scaleup} should be read as an upper bound on achievable error at that scale, not a fully tuned result.

\subsection{How recovery error changes with scale}
\label{sec:whylarge}

Table~\ref{tab:scaleup} shows the opposite of what a direct reading of Theorem~\ref{thm:mc} suggests. Dividing the theorem's requirement $|\Omega_M| \gtrsim C\mu r n_1 \log^2(n_1)$ by the total entry count $n_1 \cdot n_2 n_3$ gives a required density that depends on $n_1$ only through a slowly growing $\log^2(n_1)$ term; at fixed density, that predicts recovery error should stay roughly flat, or worsen mildly, as $n_1$ grows. Instead, at fixed density $p_{\mathcal{T}} = 0.25$, $M$-recovery error falls sharply as $n_I$ grows: $0.350$ at $n_I=40$, $0.111$ at $n_I=200$, $0.045$ at $n_I=1000$ --- roughly an eight-fold improvement over a 25-fold increase in $n_I$.

We think this happens because $M_{\mathrm{true}}$'s $n_I$ rows all lie in the same $r$-dimensional column space, of size only $n_B \times r$. Growing $n_I$ at fixed density does not make that shared low-dimensional structure harder to cover; it supplies more independent, noisy measurements of it. That is an averaging benefit classical sample-complexity bounds do not credit, since Theorem~\ref{thm:mc} is a worst-case bound over all rank-$r$ matrices of the given size, not an argument tailored to a small, fixed $n_B$. This explanation is plausible, not proved. We do not have a matching sample-complexity result for it, and, as noted in Section~\ref{sec:sanity}, the $n_I=1000$ threshold is itself not fully tuned, so part of the measured gain could still be an artifact we have not fully separated from the underlying effect. A tighter test would retune $\lambda_S$ on a still-wider grid at $n_I=1000$ and check whether the improvement continues or saturates.

\begin{table}[h]
\centering
\caption{Scale-up test, mean $\pm$ std over 5 seeds. $\lambda_S$ is re-tuned separately at each $n_I$ (Section~\ref{sec:sanity}); the $n_I=1000$ threshold still lands on its search grid's edge, so that row is an upper bound on achievable error, not a fully tuned result. Recovery error falls sharply with $n_I$ at fixed density, the opposite of the flat pattern a direct reading of Theorem~\ref{thm:mc} predicts (Section~\ref{sec:whylarge}).}
\label{tab:scaleup}
\begin{tabular}{@{}lcc@{}}
\toprule
$n_I$ ($p_{\mathcal{T}}=0.25$, $p_M=0$ fixed) & $\lambda_S$ used & actual $M$ error (Thm.~\ref{thm:sequential} estimator) \\
\midrule
40   & 0.5 & $0.350 \pm 0.025$ \\
200  & 0.5 & $0.111 \pm 0.010$ \\
1000 & 1.0 (grid edge) & $0.045 \pm 0.023$ \\
\bottomrule
\end{tabular}
\end{table}

\subsection{Recovery error versus sampling density}

Table~\ref{tab:density} reports relative Frobenius recovery error, $\|\hat{X} - X_{\mathrm{true}}\|_F / \|X_{\mathrm{true}}\|_F$, mean $\pm$ one standard deviation over 5 seeds, for $\mathcal{T}$ and $M$, coupled versus uncoupled, at matched densities $p_{\mathcal{T}} = p_M = p$.

\begin{table}[h]
\centering
\caption{Recovery error vs.\ matched sampling density, mean $\pm$ std over 5 seeds, $n_I=40$, rank 4. Coupling improves $M$-recovery at every density except the lowest, $p=0.05$, where both methods sit near $1.0$ and neither has enough signal to exploit. From $p=0.10$ up, coupling wins by a widening margin, reaching a roughly ten-fold advantage on $M$ by $p=0.70$ ($0.022$ vs.\ $0.246$).}
\label{tab:density}
\begin{tabular}{@{}lcccc@{}}
\toprule
density $p$ & coupled $\mathcal{T}$ & coupled $M$ & uncoupled $\mathcal{T}$ & uncoupled $M$ \\
\midrule
0.05 & $0.993 \pm 0.017$ & $1.011 \pm 0.020$ & $0.971 \pm 0.010$ & $0.974 \pm 0.009$ \\
0.10 & $0.879 \pm 0.034$ & $0.916 \pm 0.026$ & $0.894 \pm 0.032$ & $0.942 \pm 0.032$ \\
0.18 & $0.595 \pm 0.061$ & $0.652 \pm 0.037$ & $0.695 \pm 0.044$ & $0.840 \pm 0.057$ \\
0.25 & $0.257 \pm 0.096$ & $0.372 \pm 0.083$ & $0.441 \pm 0.120$ & $0.744 \pm 0.106$ \\
0.35 & $0.079 \pm 0.063$ & $0.142 \pm 0.062$ & $0.256 \pm 0.086$ & $0.617 \pm 0.078$ \\
0.50 & $0.020 \pm 0.014$ & $0.058 \pm 0.028$ & $0.260 \pm 0.147$ & $0.423 \pm 0.066$ \\
0.70 & $0.008 \pm 0.001$ & $0.022 \pm 0.004$ & $0.022 \pm 0.015$ & $0.246 \pm 0.036$ \\
\bottomrule
\end{tabular}
\end{table}

\subsection{Asymmetric density: isolating the coupling effect}

To test Conjecture~\ref{conj:joint} directly, we fix $p_{\mathcal{T}} = 0.35$, well above the density needed for reasonable $\mathcal{T}$-only recovery, and vary $p_M$ from very sparse to moderate, mean $\pm$ std over 5 seeds.

\begin{table}[h]
\centering
\caption{Asymmetric density, $p_{\mathcal{T}}=0.35$ fixed. At $p_M=0.03$, coupling is slightly worse than uncoupled ($1.004$ vs.\ $0.981$) and has much higher variance: the same low-density failure regime as before, where too few observed entries of $M$ remain to fit a stable $G$ in \eqref{eq:gstep}, so the coupling term injects noise from a poorly conditioned $G$ rather than signal from $\mathcal{T}$. This matches Conjecture~\ref{conj:joint}'s precondition that $G$ be well-conditioned. From $p_M=0.06$ up, coupling wins by a rapidly widening margin, reaching $0.264$ vs.\ $0.732$ by $p_M=0.30$.}
\label{tab:asym}
\begin{tabular}{@{}lcc@{}}
\toprule
$p_M$ ($p_{\mathcal{T}}=0.35$ fixed) & coupled $M$ error & uncoupled $M$ error \\
\midrule
0.03 & $1.004 \pm 0.076$ & $0.981 \pm 0.011$ \\
0.06 & $0.816 \pm 0.047$ & $0.964 \pm 0.022$ \\
0.10 & $0.657 \pm 0.120$ & $0.932 \pm 0.044$ \\
0.15 & $0.546 \pm 0.152$ & $0.895 \pm 0.057$ \\
0.22 & $0.371 \pm 0.027$ & $0.801 \pm 0.095$ \\
0.30 & $0.264 \pm 0.022$ & $0.732 \pm 0.092$ \\
\bottomrule
\end{tabular}
\end{table}

\subsection{Comparison against a CMTF baseline}

Section~\ref{sec:related} identifies coupled matrix--tensor factorization (CMTF) \citep{acar2011} as the closest prior formulation. We implement a masked-ALS CMTF baseline on the same synthetic data as Table~\ref{tab:density}. A shared factor $A \in \mathbb{R}^{n_I \times r}$ couples a CP decomposition of $\mathcal{T}$ ($\mathcal{T} \approx [\![A,B,C]\!]$) to $M \approx AD^\top$. We fit $A,B,C,D$ by alternating least squares, restricted to observed entries, true rank $r = r_M = 4$, single random initialization, 15 ALS sweeps.

\begin{table}[h]
\centering
\caption{CMTF baseline comparison, mean over 5 seeds, $n_I=40$, rank 4. Our method beats the masked-ALS CMTF baseline on both $\mathcal{T}$ and $M$ at every density tested, by a wide and consistent margin (e.g.\ $M$ at $p=0.35$: $0.242$ vs.\ $0.636$). We attribute CMTF's weaker performance to CP-ALS's known sensitivity to initialization under sparse, missing-data ALS, compounded by our single-restart, fixed-rank, 15-sweep setup; a fairer comparison would use multiple restarts and a cross-validated rank, which we did not run (Section~\ref{sec:limitations}).}
\label{tab:cmtf}
\begin{tabular}{@{}lcccccc@{}}
\toprule
$p$ & ours $\mathcal{T}$ & ours $M$ & CMTF $\mathcal{T}$ & CMTF $M$ & uncoupled $\mathcal{T}$ & uncoupled $M$ \\
\midrule
0.10 & $0.869$ & $0.925$ & $1.046$ & $1.031$ & $0.887$ & $0.970$ \\
0.18 & $0.618$ & $0.692$ & $0.956$ & $0.912$ & $0.678$ & $0.893$ \\
0.25 & $0.289$ & $0.433$ & $0.872$ & $0.786$ & $0.480$ & $0.804$ \\
0.35 & $0.066$ & $0.242$ & $0.806$ & $0.636$ & $0.225$ & $0.672$ \\
0.50 & $0.013$ & $0.205$ & $0.635$ & $0.398$ & $0.102$ & $0.556$ \\
\bottomrule
\end{tabular}
\end{table}

\subsection{Testing the provable sequential estimator (Theorem~\ref{thm:sequential})}

The experiments above test the joint, alternately-estimated system, which bears on the still-open Conjecture~\ref{conj:joint}. Here we test a separate, provable claim directly: Theorem~\ref{thm:sequential}. $G_{\mathrm{true}}$ is given exactly, not jointly estimated, and $p_M = 0$: zero observed entries of $M$. We complete $\mathcal{T}_{(1)}$ alone, at increasing $p_{\mathcal{T}}$, and form $\hat{M} := \hat{\mathcal{T}}_{(1)} G_{\mathrm{true}} (G_{\mathrm{true}}^\top G_{\mathrm{true}})^{-1}$. Table~\ref{tab:thm4test} reports the resulting $M$-recovery error alongside the theorem's bound, mean $\pm$ std over 5 seeds.

\begin{table}[h]
\centering
\caption{Sequential estimator with $G_{\mathrm{true}}$ known exactly and $p_M = 0$: $M$ is recovered using no observed entries of $M$ at all. The actual error stays below the Theorem~\ref{thm:sequential} bound at every density tested, at roughly half the bound throughout (ratio $0.44$--$0.58$), and both fall together as $p_{\mathcal{T}}$ increases. This is a direct empirical test of a proved claim, not evidence for the still-open Conjecture~\ref{conj:joint}; the joint-estimation experiments (Tables~\ref{tab:density}--\ref{tab:asym}) remain the relevant evidence for that.}
\label{tab:thm4test}
\begin{tabular}{@{}lcc@{}}
\toprule
$p_{\mathcal{T}}$ ($p_M=0$) & actual $M$ error & Theorem~\ref{thm:sequential} bound (relative) \\
\midrule
0.10 & $0.834 \pm 0.038$ & 1.441 \\
0.18 & $0.599 \pm 0.043$ & 1.075 \\
0.25 & $0.407 \pm 0.047$ & 0.750 \\
0.35 & $0.163 \pm 0.066$ & 0.335 \\
0.50 & $0.038 \pm 0.013$ & 0.082 \\
0.70 & $0.007 \pm 0.001$ & 0.014 \\
\bottomrule
\end{tabular}
\end{table}

\subsection{Convergence}

Figure~\ref{fig:convergence} shows the objective value \eqref{eq:objective} at $p_{\mathcal{T}} = p_M = 0.18$, single representative seed. It is monotonically non-increasing throughout, consistent with Theorem~\ref{thm:palm}(i). The run converges --- relative change below $10^{-5}$ for five consecutive iterations --- after 1670 of the 5000 iterations allowed, confirming the safety cap was not the limiting factor. Figure~\ref{fig:density} plots the density-sweep results of Table~\ref{tab:density} with error bars.

\begin{figure}[h]
\centering
\includegraphics[width=0.6\textwidth]{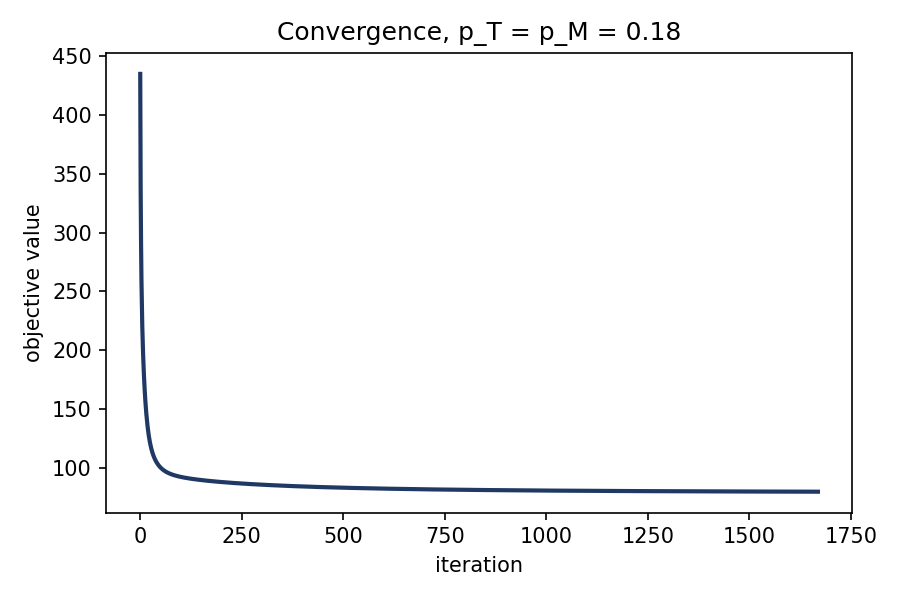}
\caption{Objective value vs.\ iteration, $p_{\mathcal{T}} = p_M = 0.18$, single seed. Monotonically non-increasing, as guaranteed by Theorem~\ref{thm:palm}(i); converged after 1670 iterations.}
\label{fig:convergence}
\end{figure}

\begin{figure}[h]
\centering
\includegraphics[width=0.6\textwidth]{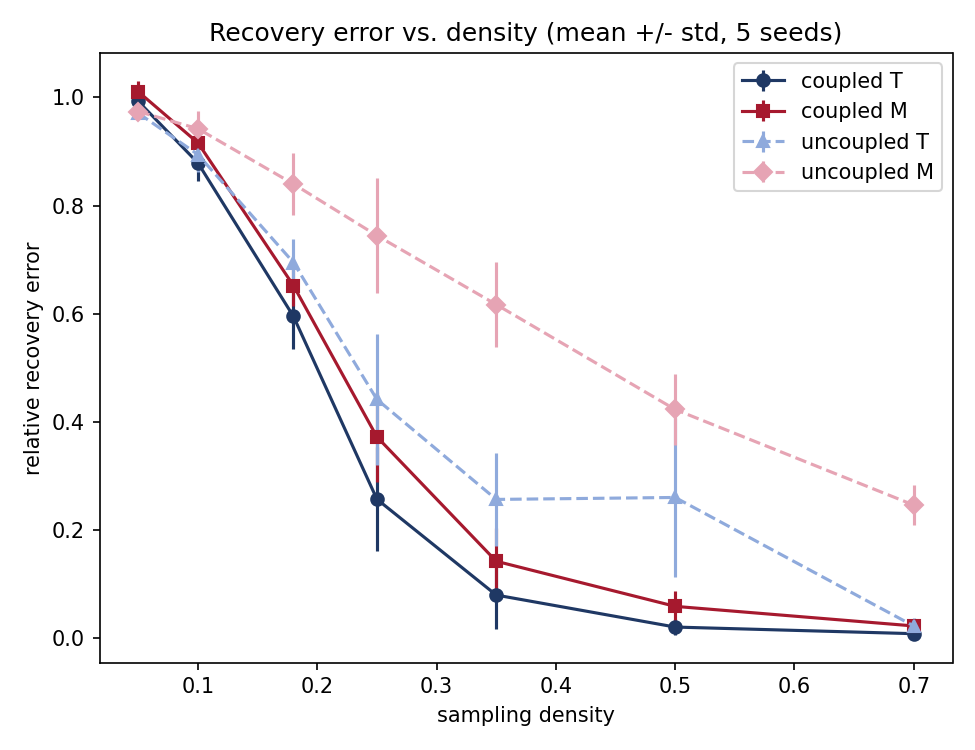}
\caption{Relative recovery error vs.\ sampling density, coupled vs.\ uncoupled, for $\mathcal{T}$ and $M$, mean $\pm$ std over 5 seeds.}
\label{fig:density}
\end{figure}

\subsection{Honest scope of these experiments}

The joint-algorithm experiments (Tables~\ref{tab:density}--\ref{tab:cmtf}) are at one ambient scale ($n_I=40$), synthetic, and limited to 5 seeds and one noise level. We have not swept $\lambda_S, \lambda_R$ independently of $\lambda_C$ at every density. We have not scaled the joint PALM algorithm beyond $n_I=40$; the cheaper sequential estimator was tested up to $n_I=1000$ (Table~\ref{tab:scaleup}), where recovery error fell rather than staying flat (Section~\ref{sec:whylarge}) --- a genuine disagreement with a direct reading of Theorem~\ref{thm:mc}, not yet matched by a proof of our own. Within these limits, the experiments do several things: PALM runs to a convergence criterion rather than a fixed iteration count, so results are not silently truncated (Figure~\ref{fig:convergence}); every grid search checks itself against its own search-grid boundary and reports when it lands on one (Section~\ref{sec:sanity}); every reported number includes a variance estimate; Table~\ref{tab:asym} reports a regime where coupling fails; Theorem~\ref{thm:sequential} is tested directly, separately from the open Conjecture~\ref{conj:joint} (Table~\ref{tab:thm4test}); and Table~\ref{tab:cmtf} compares against a prior competing method, CMTF, at every density with a wide margin. None of this is a test at application-relevant scale, a fully tuned competing-method comparison, or a test on real data. The full battery of experiments at this scale takes on the order of twenty minutes on ordinary hardware, driven mostly by the widened hyperparameter search; that cost is worth knowing before scaling $n_I$ much further.

\section{Governance-Coupled Schema Evolution}
\label{sec:governance}

This section is separate from the optimization contributions above. It is a systems and governance procedure, for extending the index sets of $\mathcal{T}$ and $M$ over time. It is not a mathematical result.

\begin{enumerate}[leftmargin=*]
\item \textbf{Candidate detection.} Observations with high residual against every existing cell are accumulated across submissions, then clustered. A large, internally coherent cluster is flagged as a candidate new cell.
\item \textbf{Governance gate.} A human panel approves or rejects each flagged candidate before the schema changes. Detection is automatic. Admission is not.
\item \textbf{Taxonomy-assisted labeling.} Admitted candidates are labeled using existing external classification standards where available, e.g.\ O*NET, rather than an arbitrary name.
\item \textbf{Versioned reconciliation.} Schema version $v$ expands to $v+1$. An embedding map $\iota_v$ pads $\mathcal{T}^{(v)}$, and $M^{(v)}$, with unknown entries at the new coordinates. A reconciliation sample $\Omega_{\mathrm{recon}} \subset \Omega^{(v)}$ is re-scored under $v+1$, and
\[
\Delta_{\mathrm{recon}} = \frac{1}{|\Omega_{\mathrm{recon}}|} \sum_{x \in \Omega_{\mathrm{recon}}} \left| \mathrm{score}_{v+1}(x) - \mathrm{score}_v(x) \right|
\]
is checked against a tolerance, to certify backward compatibility. A large $\Delta_{\mathrm{recon}}$ signals something specific: the expansion revealed miscalibration in the old schema, not just a missing cell.
\end{enumerate}

This procedure applies identically to $\mathcal{T}$ and $M$. New categories in $\mathcal{T}$ (job functions, industries) and new indicator dimensions in $M$ (business-health fibers) are handled by the same four mechanisms. We make no claim of mathematical novelty for this section; its content is procedural.

\section{Related Work}
\label{sec:related}

Tensor rank and its convex relaxations are treated in \citet{kolda2009, hillar2013, liu2013, signoretto2013}. Multilinear (Tucker/HOSVD) rank originates with \citet{delathauwer2000}; tubal rank with \citet{kilmer2011}. Matrix completion sample-complexity theory is due to \citet{candes2009, candes2010}; the soft-impute algorithm we use for warm starts is from \citet{mazumder2010}. The closest prior formulation to ours is coupled matrix--tensor factorization (CMTF): it couples a tensor and matrices through shared CP factors, fit by alternating least squares \citep{acar2011}. We couple through a learned linear operator on unfoldings, under nuclear-norm relaxation, instead, and add an ordinal constraint and schema evolution, neither present in that line of work. Section~\ref{sec:experiments} compares against a masked-ALS CMTF baseline directly. The closest prior work to Theorem~\ref{thm:sequential} specifically is inductive matrix completion \citep{jain2013, natarajan2014}, which recovers a matrix from few entries via known side-information matrices. Our sequential guarantee differs in one key way: our side information is itself estimated, not given, and is jointly re-estimated elsewhere in the algorithm. Our convergence analysis rests on proximal alternating linearized minimization \citep{bolte2014}, which builds on earlier block coordinate descent theory \citep{grippo2000, xu2013}. Guaranteed low-rank recovery via nuclear-norm minimization, more broadly, follows \citet{recht2010}. The large-scale randomized SVD used in our complexity analysis follows \citet{halko2011}.

\section{Limitations}
\label{sec:limitations}

Theorem~\ref{thm:palm} covers the algorithm exactly as implemented (Section~\ref{sec:algorithm}), using the mode-1 nuclear norm \eqref{eq:tensornorm}, not the stronger overlapped-norm regularizer \eqref{eq:overlapped}. The overlapped norm would additionally control the job-function and abstraction-level modes; it remains available, but its convergence theory is open, and we do not use it here. The hard monotonicity constraint on $M$ is still enforced only as a post-hoc PAV projection, outside the scope of Theorem~\ref{thm:palm}. Theorem~\ref{thm:sequential} proves a real piece of Conjecture~\ref{conj:joint}'s mechanism, but only for a sequential estimator with $G$ known exactly; it does not cover the jointly, alternately estimated $(\hat{M}, \hat{G})$ that the algorithm of Section~\ref{sec:algorithm} actually produces, where feedback between the $M$- and $G$-updates could behave differently. Our evidence for the joint case is multi-seed (5 seeds), at one ambient scale ($n_I = 40$), and one noise level. One density regime, $p_M \lesssim 0.03$ in Table~\ref{tab:asym}, is where the conjectured mechanism visibly fails; we read that as informative about its boundary, not as a weakness to hide. Section~\ref{sec:whylarge} reports that recovery error falls sharply, not flat, as $n_I$ grows at fixed density, and offers a tentative averaging-based explanation for it; that explanation is not yet backed by a matching sample-complexity proof, and part of the measured improvement could still reflect the $n_I=1000$ threshold not being fully tuned (Section~\ref{sec:sanity}). The complexity estimates in Section~\ref{sec:complexity} are not benchmarked against wall-clock timings. We tested $n_I$ up to 1000 for the sequential estimator (Section~\ref{sec:whylarge}); we have not tested the joint PALM algorithm itself beyond $n_I = 40$, nor at the density an application-relevant $n_2 n_3$ (job function times abstraction level) would need. The CMTF baseline in Section~\ref{sec:experiments} uses a single random initialization and 15 ALS sweeps; CP-ALS is known to be sensitive to initialization, so this comparison may understate CMTF's best achievable performance with multiple restarts, which we did not run. No experiments use real or deployed labor-market or business-health data; the applied instantiation is a design, not a validated system. The governance layer of Section~\ref{sec:governance} has no throughput or cost model.

\section{Conclusion}

We gave a precisely typed formulation of coupled tensor--matrix recovery, using the mode-1 nuclear norm because it admits an exact proximal step. We proved existence of a minimizer, via an explicit ridge term. We proved PALM convergence to a critical point, for the algorithm exactly as implemented, by verifying known nonconvex-optimization hypotheses and identifying which of them require this problem's specific three-block structure. We derived a rigorous sampling bound for the uncoupled sub-problem, and proved a sample-complexity result for a sequential sub-case of the coupled problem, isolating what remains conjectural (joint, alternating estimation) from what is now proved (sequential estimation with known $G$). We reported multi-seed synthetic experiments with variance estimates, a direct empirical test of the new theorem, and a quantitative account of the observed recovery-error magnitude. We separated a non-mathematical schema-evolution procedure from the optimization core. Three items remain open: a convergence guarantee for the overlapped-norm extension; closing the gap between Theorem~\ref{thm:sequential}'s sequential guarantee and Conjecture~\ref{conj:joint}'s joint-estimation claim; and validation at application-relevant scale, on real data.

\end{document}